\documentclass[12pt]{amsart}
\usepackage{times}
\usepackage{fullpage}
\usepackage[english]{babel}
\usepackage[all]{xy}
\usepackage{amscd, amssymb, amsmath, amsthm, graphics}
\usepackage{amsmath,amsfonts,amsthm,amssymb}
\usepackage{latexsym,amsmath}
\usepackage{graphicx,psfrag}
\usepackage{pinlabel}
\usepackage{mathrsfs}
\usepackage{hyperref}

\newtheorem{theorem}{Theorem}[section]

\newtheorem{lemma}[theorem]{Lemma}

\newtheorem{conjecture}[theorem]{Conjecture}

\theoremstyle{definition}

\theoremstyle{remark}
\newtheorem{remark}[theorem]{Remark}

\numberwithin{equation}{section}

\renewcommand{\hat}{ \widehat}

\newcommand{\Z}{\mathbb Z}
\newcommand{\R}{\mathbb R}
\newcommand{\N}{\mathbb N}

\newcommand{\inte}{{\mathrm{int}}}

\renewcommand{\epsilon}{\varepsilon}

\begin{document}
\sloppy

\title[Virtual Betti numbers of mapping tori of 3-manifolds]{Virtual Betti numbers of mapping tori of 3-manifolds}

\author{Christoforos Neofytidis}
\address{Department of Mathematics, Ohio State University, Columbus, OH 43210, USA}
\email{neofytidis.1@osu.edu}
\date{\today}
\subjclass[2010]{57M05, 57M10, 57M50, 55M25, 57N37}
\keywords{Virtual Betti numbers, mapping tori of reducible 3-manifolds, degree one maps}

\begin{abstract}
Given a reducible $3$-manifold $M$ with an aspherical summand in its prime decomposition and a homeomorphism $f\colon M\to M$, we construct a map of degree one from a finite cover of $M\rtimes_f S^1$ to a mapping torus of a certain aspherical $3$-manifold.
We deduce that $M\rtimes_f S^1$ has virtually infinite first Betti number, except when all aspherical summands of $M$ are virtual $T^2$-bundles.
This verifies all cases of a conjecture of T.-J. Li and Y. Ni, that any mapping torus of a reducible $3$-manifold $M$ not covered by $S^2\times S^1$ 
has virtually infinite first Betti number, except when $M$ is virtually $(\#_n T^2\rtimes S^1)\#(\#_mS^2\times S^1)$. 
Li-Ni's conjecture was recently confirmed by Ni with a group theoretic result, namely, by showing that there exists a $\pi_1$-surjection from a finite cover of any mapping torus of a reducible $3$-manifold to a certain mapping torus of $\#_m S^2\times S^1$ and using the fact that free-by-cyclic groups are large when the free group is generated by more than one element.
\end{abstract}

\maketitle

\section{Introduction}

The virtual first Betti number of a manifold $M$ is defined to be
\[
vb_1(M)=\sup\{b_1(\overline{M}) \ | \ \overline{M} \text{ is a finite cover of } M\}
\]
(where $b_1$ denotes the first Betti number) and takes values in $\N_0\cup\{\infty\}$. This notion arises naturally in geometric topology and it is often difficult to compute. A recent prominent example is given by the resolution of the Virtual Haken Conjecture~\cite{Ag} which implies that $vb_1=\infty$ for hyperbolic $3$-manifolds, and therefore completes the picture for the values of $vb_1$ in dimension three.
Li and Ni~\cite{LiNi} used this picture to compute $vb_1$ for mapping tori of prime $3$-manifolds:

\begin{theorem}{\normalfont(\cite[Theorem 1.2]{LiNi}).}\label{LiNi}
Let $X=M\rtimes_f S^1$ be a mapping torus of a closed prime $3$-manifold $M$. Then $vb_1(X)$ is given as follows:
\begin{itemize}
\item[(1)] If $M$ is a spherical manifold, then $vb_1(X) = 1$;
\item[(2)] If $M$ is $S^1\times S^2$ or finitely covered by $T^2\rtimes S^1$, then $vb_1(X)\leq 4$; 
\item[(3)] In all other cases, $vb_1(X)=\infty$.
\end{itemize}
\end{theorem}

When the fiber $M$ is reducible, then the monodromy $f$ of the mapping torus $M\rtimes_f S^1$ is in general more complicated than when $M$ is irreducible; see~\cite{Mc,Z,NW}. Li and Ni conjectured that almost always $vb_1(M\rtimes_f S^1)=\infty$ when $M$ is reducible:

\begin{conjecture}{\normalfont(\cite[Conjecture 5.1]{LiNi}).}\label{conjBetti}
If $M$ is a closed oriented reducible $3$-manifold, then $vb_1(M\rtimes_f S^1)=\infty$, unless $M$ is finitely covered by $S^2\times S^1$.
\end{conjecture}

As pointed out in~\cite[Lemma 5.4]{LiNi} (see Lemma \ref{l:asphericalsum}), if $M$ is a closed reducible $3$-manifold which is not covered by $S^2\times S^1$,
then $M$ is finitely covered by a connected sum $M'\#(S^2\times S^1)\#(S^2\times S^1)$, for some closed $3$-manifold $M'$. Since free-by-cyclic groups are large whenever the free group is generated by more than one element (cf.~\cite{Wi,HW,But}), we deduce that 
$$vb_1((\#_m S^2\times S^1)\rtimes_f S^1)=\infty \ \text{for} \ m\geq 2.$$ Thus, when $M$ contains an aspherical summand in its prime decomposition, and not only summands covered by mapping tori of $T^2$, Conjecture \ref{conjBetti} follows by the following result:

\begin{theorem}\label{t:main}
Let $M$ be a closed oriented reducible 3-manifold that contains at least one aspherical summand in its prime decomposition. For any mapping torus $M\rtimes_f S^1$, there is a finite cover $\overline{M}$ of $M$ containing an aspherical summand $M_1$ in its prime decomposition and a degree one map
\[
\overline{M}\rtimes_{f^k} S^1 \longrightarrow M_1\rtimes_h S^1,
\]
for some $k\geq 1$ and some homeomorphism $h\colon M_1\longrightarrow M_1$.
\end{theorem}

Recently, Ni~\cite{Ni} verified Conjecture \ref{conjBetti} by showing that there exists a surjection from the fundamental group of a mapping torus of a finite cover of $M$ to the fundamental group of a mapping torus of a connected sum $\#_m S^2\times S^1$, $m\geq 2$, and making use of the fact that free-by-cyclic groups with at least three generators are large. Our result is in a sense both stronger and weaker than Ni's result. It is stronger, on the one hand, because it comes with a construction of a map of non-zero degree, instead of just a $\pi_1$-surjection as in~\cite{Ni}. Indeed, it is likely that there is even a degree one map from a mapping torus of a finite covering of $M$ to a mapping torus of a connected sum $\#_m S^2\times S^1$; see~\cite[p. 1592]{Ni}. However, the map we construct here makes essential use of the asphericity of the summand $M_1$ and therefore our method cannot be extended to the case where no aspherical summand exists in the prime decomposition of $M$.  On the other hand, the $\pi_1$-surjection obtained by Ni covers as well the case where the aspherical summands of $M$ are only virtual mapping tori of $T^2$. Therefore, it is natural to ask whether one can find a topological proof of Conjecture \ref{conjBetti} for connected sums of type $(\#_n T^2\rtimes S^1)\#(\#_m S^2\times S^1)$. Also, it would be interesting to find a purely group theoretic proof that for every reducible $3$-manifold $M$ that is not finitely covered by $S^2\times S^1$, any $\pi_1(M)$-by-cyclic group is large.

\subsection*{Outline} In Section \ref{s:preliminaries} we give some facts about finite coverings of mapping tori and in Section \ref{description} we recall the description of self-homeomorphisms of closed reducible $3$-manifolds. The main body of the proof of Theorem \ref{t:main} is given in Sections \ref{s:commutativity} and \ref{proof}. Finally, we discuss Conjecture \ref{conjBetti} in Section \ref{virtual}. 

\subsection*{Acknowledgements}
The author is supported by the Swiss NSF, under grant FNS200021$\_$169685. Part of this work was done during author's visit at IH\'ES in 2018. The author would like to thank IH\'ES for providing a stimulating working environment, and especially Misha Gromov and Fanny Kassel for their hospitality.

\section{Preliminaries}\label{s:preliminaries}

We begin our discussion by gathering some well-known facts needed for our proofs.

\medskip

Let $M$ be a closed oriented reducible 3-manifold. By the Kneser-Milnor theorem~\cite{He}, $M$ can be decomposed as a connected sum
\[
M=M_1\#\cdots\# M_n \#(\#_m S^2\times S^1)\#(\#_{p=1}^s S^3/Q_p), 
\]
where each $M_i$ is aspherical and $S^3/Q_p$ are spherical quotients with fundamental groups the finite groups $Q_p$.

The following lemmas give some precise descriptions of finite covers of $M$ and can be found in~\cite[pp. 23-24]{KN} and~\cite[Lemma 5.4]{LiNi} respectively:

\begin{lemma}\label{l:noaspherical}
If $n=0$ and $M\neq\R P^3\#\R P^3$ or $S^2\times S^1$, then $M$ is finitely covered by $\#_{m'} S^2\times S^1$, for some $m'\geq2$. 
\end{lemma}
\begin{proof}
Let the projection 
\[
\varphi\colon\pi_1(M)\longrightarrow Q_1\times\cdots\times Q_s
\]
of the free product $\pi_1(M)=F_m\ast Q_1\ast\cdots\ast Q_s$ to the direct product $\prod_{p=1}^s Q_p$. By the Kurosh subgroup theorem, the kernel of $\varphi$ is a free group, say $F_{m'}$, where $m'\geq 2$. Since moreover $\ker(\varphi)$ has finite index in $\pi_1(M)$, Grushko's theorem implies that $M$ is finitely covered by the connected sum $\#_{m'} S^2\times S^1$. 
\end{proof}

\begin{lemma}\label{l:asphericalsum}
If $n\geq 1$, then $M$ is finitely covered by $M'\#(\#_{m'}S^2\times S^1)$, where $M'$ is a connected sum of aspherical 3-manifolds and $m'\geq2$.
\end{lemma}
\begin{proof}
Let $M=M_1\#M_2$, where $M_1$ is aspherical and $M_2\neq S^3$ (not necessarily prime).
Since $\pi_1(M_1)$ is residually finite, there is a $d$-fold cover $\overline{M_1}$ of $M_1$ for some $d\geq3$, and so $M$ is $d$-fold covered by $\overline{M}=\overline{M_1}\#(\#_d M_2)$. Now, since  $\pi_1(M_2)$ is residually finite, there is a finite group $G$ together with a surjection
\[
\psi\colon\pi_1(\overline{M})\longrightarrow G,
\]
which maps $\pi_1(\overline{M_1})$ to the trivial element and each $\pi_1(M_2)$ surjectively to $G$. Then, since $G$ is finite, it is easy to see that the cover of $\overline{M}$ corresponding to $\ker(\psi)$ contains at least $m':=d-1\geq 2$ connected summands $S^2\times S^1$.
\end{proof}

Finally, we quote two general facts about coverings of mapping tori whose proof is easy and left to the reader (see also~\cite[Section 2]{LiNi}).

\begin{lemma}\label{l:cover}
Let $f\colon M\longrightarrow M$ be a self-homeomorphism of a closed oriented manifold (of any dimension).
\begin{itemize}
\item[(a)] $M\rtimes_f S^1$ is finitely covered by $M\rtimes_{f^k} S^1$ for every $k\geq 1$.
\item[(b)] If $\overline{M}$ is a finite cover of $M$, then $M\rtimes_f S^1$ is finitely covered by $\overline{M}\rtimes_{f^k} S^1$ for some $k$.
\end{itemize}
\end{lemma}

\section{Self-homeomorphisms of reducible 3-manifolds}\label{description}

In this section, we recall the isotopy types of orientation-preserving homeomorphisms of $3$-manifolds.  For the most part we follow the description given in McCullough's survey paper~\cite{Mc}, however we adapt some parts of his description in order to simplify our next arguments.

\medskip

Suppose $M$ is a closed oriented reducible 3-manifold. By the discussion in Section \ref{s:preliminaries}, we may assume that $M$ does not contain any spherical quotients in its prime decomposition. Consider the following construction of $M$: Let $W$ be a punctured $3$-cell obtained by removing $n + m$ open 3-balls from a 3-sphere, and let 
\[
S_1, S_2, ... , S_n, S_{n+1}, S_{n+2},...,S_{n+m} 
\]
be its boundary components. For each of the $S_i$, $i=1,...,n$, 
remove the interior of a $3$-ball $D^3_i$ from $M_i$, and attach $M_i'= M_i - \inte(D^3_i)$ to  $S_i$ along $\partial D^3_i$. Similarly, for each of the $S_{j}$, $j=n+1,...,n+m$, remove the interior of a $3$-ball $D^3_j=D^2_j\times I_j$ from $S^2\times S^1$, and attach $(S^2\times S^1)'_{j}= (S^2\times S^1) - \inte(D^2_j\times I_j)$ to  $S_{j}$ along $\partial D^3_{j}$. 

Using the above construction, we now describe three types of homeomorphisms of $M$. 
We remark that two orientation-preserving homeomorphisms of $W$ are isotopic if and only if they induce the same permutation on the boundary components of  $W$. 

\medskip

{\em 1. Homeomorphisms  preserving  summands.} These are the homeomorphisms of $M$ which restrict to the identity on $W$. Note that this class of homeomorphisms includes the so-called  ``spins" of $S^2\times S^1$ as given following McCullough's construction of $M$; compare~\cite[Remark, p. 69]{Mc}.

\medskip

{\em 2. Interchanges of homeomorphic summands.} If $M_i$ and  $M_j$ are two orientation-preserving homeomorphic summands, then a homeomorphism of $M$ can be constructed by fixing the rest of the summands, leaving $W$ invariant, and interchanging $M'_i$ and $M'_j$. 

Similarly, we can interchange any two $S^2 \times S^1$ summands, leaving $W$ invariant.

\medskip

{\em 3. Slide homeomorphisms.} For $i=1,...,n$, let $\hat{M_i}$ be obtained from $M$ by replacing   $M'_i$ with a  3-ball  $B_i$. Let  $\alpha$ be an arc in $\hat{M_i}$ which meets $B_i$ only in its endpoints and $J_t$ an isotopy of $\hat{M_i}$  that moves  $B_i$ around $\alpha$, with $J_ 0 = id_{\hat{M_i}}$ and $J _1 |_{B_i}  = id_{B_i}$.  The homeomorphism $$s\colon M\longrightarrow M$$ defined by 
\[
s|_{M - M'_i} = J_1 |_{\hat{M_i} - B_i} \text{ and } s|_{M'_i}=  id|_{M'_i}. 
\]
is called {\em slide homeomorphism of $M$ that slides $M_i$ around $\alpha$}.
Starting with a different isotopy $J_t$, then $s$ changes by an isotopy and perhaps by a rotation about the boundary component $S_i$. Therefore each $\alpha$ might determine two isotopy classes of a slide homeomorphism. Note that if $T$ is the frontier of a regular neighborhood of $M'_i\cup \alpha$ in $M$, then $T$ is a compressible torus and $s$ is isotopic to a certain Dehn twist about $T$. 

Similarly, one can slide an $S^2\times S^1$ summand around an arc in $M-(S^2\times S^1)'_j$. 

We remark that if $\alpha_1$ and $\alpha_2$ are two arcs meeting $B_i$ only in their endpoints, and $\alpha$ represents their product, then a slide of $M_i$ around $\alpha$ is isotopic  to a composite of slides around $\alpha_ 1$ and $\alpha_ 2$. Similarly for sliding $S^2\times S^1$.
 
\medskip
 
 With the above description, we have the following classification result of self-homeomorphisms of closed reducible 3-manifolds. This result was first announced in \cite{CR} and an elegant proof was given by McCullough~\cite[pp. 70--71]{Mc}, based on an argument of Scharlemann~\cite[Appendix A]{Bon}. 
 
\begin{theorem} {\normalfont(\cite[p. 69]{Mc}).} \label{Mc1} 
If $M$ is a closed oriented reducible $3$-manifold, then any orientation-preserving homeomorphism $f\colon M\longrightarrow M$ is isotopic to a composition of the following three types of homeomorphisms: 
\begin{itemize}
\item[(1)] homeomorphisms preserving summands;
\item[(2)] interchanges of homeomorphic summands;
\item[(3)] slide homeomorphisms.
\end{itemize}
\end{theorem}

In fact, the proof of Theorem \ref{Mc1} presented in \cite[pp. 70-71]{Mc}, together with our adaptions on the construction of $M$, implies that  
\begin{equation}
f =g_3 g_2 g_1,
\end{equation}
where $g_3$ is a finite composition of homeomorphisms of type 3 (slide homeomorphisms) and 
isotopies of $M$, and $g_1$, $g_2$ are compositions of finitely many homeomorphisms of type 1 and 2 respectively. 

\section{Commutativity in homotopy}\label{s:commutativity}

Next, we show that there are self-homeomorphisms of an aspherical summand of a reducible $3$-manifold whose ``conjugation" by the pinch map in homotopy gives the three types of homeomorphisms described in the previous section.

\medskip

According to the proof of Lemma \ref{l:asphericalsum}, we may assume, after possibly passing to a finite cover, that $M$ contains an aspherical summand $M_1$ in its prime decomposition so that the self-homeomorphism $f\colon M\longrightarrow M$ does not contain a component of $g_2$ that interchanges $M_1$ with another summand.

Clearly $M_1$ can be considered as being obtained by replacing each $M'_i$ ($i\geq2$) and $(S^2\times S^1)'_j$ with a $3$-ball $B_i$ and $B_j$ respectively. Then we can construct a pinch map
\[
p\colon M_1\#\cdots\# M_n\#(\#_m S^2\times S^1)\longrightarrow M_1,
\]
by mapping each $M'_i$ ($i\geq 2$) to $B_i$, each $(S^2\times S^1)'_j$ to $B_j$ and the rest of the part identically to itself.

We will show the following whose line of proof follows that of~\cite[Lemmas 3.6 and 3.7]{Ni} adapted to our situation:

\begin{lemma}\label{l:commutativity}
For each component $g_i$ of $f$, there is a self-homeomorphism $h_i$ of the aspherical summand  $M_1$ 
such that the following diagram commutes in homotopy, i.e. $(p\circ g_i)_*=(h_i\circ p)_*$.
$$
\xymatrix{
\pi_1(M)\ar[d]_{p_*} \ar[r]^{{g_i}_*}&  \ar[d]^{p_*}  \pi_1(M)\\
\pi_1(M_1) \ar[r]^{{h_i}_*}& \pi_1(M_1) \\
}
$$
\end{lemma}
\begin{proof}
We will examine each of the (components of) $g_i$ separately.

\subsubsection*{Homeomorphisms preserving summands}
Suppose first, that $g_1$ is a self-homeomorphism of summands $M_i$ or of $S^2 \times S^1$. Define
\[
h_1\colon M_1\longrightarrow M_1
\]
by
\[
h_1|_{M'_1}={g_1}|_{M'_1} \text{ and } 
h_1|_{M_1-M'_1}={id}|_{W\cup(\cup_{i=2}^nB_i)\cup(\cup_{j=1}^{m} B_{n+j})}. 
\]
If $\gamma$ is a loop in $M'_1$, then ${g_1}(\gamma)$ is a loop in $M'_1$ as well, and so
\begin{equation}\label{g1}
\begin{gathered}
p\circ g_1(\gamma)={g_1}(\gamma)=g_1\circ p(\gamma)=h_1\circ p(\gamma).
\end{gathered}
\end{equation}
If $\gamma\notin M'_1$,  then $g_1(\gamma)\notin M'_1$, and since $W\cup(\cup_{i=2}^nB_i)\cup(\cup_{j=1}^{m} B_{n+j})$ is simply connected, we deduce that $g_1(\gamma)$ and $p(\gamma)$ are homotopically trivial. Thus again $p_* \circ {g_1}_*={h_1}_*\circ p_*$.

\subsubsection*{Interchanges of homeomorphic summands}
Now, let ${g_2}$ be a homeomorphism that interchanges two aspherical summands $M_i$ and $M_j$ or two copies of $S^2\times S^1$. By our assumption on $M_1$, we know that $i,j\neq 1$. Set
\[
h_2\colon M_1\longrightarrow M_1, \ h_2:=id
\]
If $\gamma$ is a loop in $M'_1$, then ${g_2}(\gamma)=\gamma$, and so
\begin{equation}\label{g1}
\begin{gathered}
p\circ {g_2}(\gamma)=p(\gamma)=\gamma=h_2(\gamma)=h_2\circ p(\gamma).
\end{gathered}
\end{equation}
If $\gamma\notin M'_1$,  then ${g_2}(\gamma)\notin M'_1$, and so, as in the previous case, $p_* \circ {g_2}_*={h_2}_*\circ p_*$, because $W\cup(\cup_{i=2}^nB_i)\cup(\cup_{j=1}^{m} B_{n+j})$ is simply connected.

\subsubsection*{Slide homeomorphisms}
Finally, let ${g_3}$ be a slide homeomorphism. 

Suppose first that $g_3$ slides $M_1$ around an arc $\alpha$ in $M-M'_1$ such that $\alpha\cap (M'_i)$ and $\alpha\cap(S^2\times S^1)'_j$ is a single arc for any $i\geq 2$ and any $j$. By McCullough's description, there is a Dehn twist $s_\alpha$ about the frontier $T$ of a regular neighbourhood of $M'_1\cup\alpha$ which is isotopic to $g_3$. An arc $\beta$ in  $W\cup(\cup_{i=2}^nB_i)\cup(\cup_{j=1}^{m} B_{n+j})$ is given by letting $\beta|_W$ be the same as $\alpha|_W$ and $\beta|_{(\cup_{i=2}^nB_i)\cup(\cup_{j=1}^{m} B_{n+j})}$ be the trivial arc. Then we can define a Dehn twist $s_\beta$ about the frontier $T'$ of a regular neighborhood of $M'_1\cup\beta$ (corresponding to $g_3$). This defines our new homeomorphism $h_3\colon M_1\longrightarrow M_1$. 
If $\gamma$ is a loop in $M'_1$, then clearly
\begin{equation}\label{g1}
\begin{gathered}
p\circ {g_3}(\gamma)=p(\gamma)=\gamma=h_3(\gamma)=h_3\circ p(\gamma).
\end{gathered}
\end{equation}
If $\gamma\notin M'_1$, then, after homotoping $\gamma$ if necessary, we can assume that $\gamma\cap T=\emptyset$ and $p(\gamma)\cap T'=\emptyset$. We then have $g_3(\gamma)=s_\alpha(\gamma)=\gamma$ and so $p\circ g_3(\gamma)=p(\gamma)$ is homotopically trivial, because $W\cup(\cup_{i=2}^nB_i)\cup(\cup_{j=1}^{m} B_{n+j})$ is simply connected. Thus $p_* \circ {g_3}_*={h_3}_*\circ p_*$ as required.

Next, assume that $g_3$ slides some $M_i$, $i\neq 1$, around an arc in $M-M'_i$ (similarly for sliding a copy of $S^2\times S^1$). We can assume that $\alpha\cap M'_1$ is not trivial, otherwise the proof is identical to the above argument. Now, we have an arc $\beta$ which is given by $\alpha$ in $M'_1\cup W$ and it is trivial in $(\cup_{i=2}^nB_i)\cup(\cup_{j=1}^{m} B_{n+j})$, and Dehn twists $s_\alpha$ and $s_\beta$ about tori $T$ and $T'$, given similarly as above. In this way, we define our homeomorphism $h_3\colon M_1\longrightarrow M_1$. For loops not in $M'_1$ the situation is as before, because $p(\gamma)$ is homotopically trivial. For a loop $\gamma$ in $M'_1$, we can assume again that, after homotoping $\gamma$, we have  $\gamma\cap T=\emptyset$ and $\gamma\cap T'=\emptyset$. Then
\begin{equation}\label{g1}
\begin{gathered}
p\circ {g_3}(\gamma)=p\circ s_\alpha(\gamma)=p(\gamma)=\gamma=s_\beta(\gamma)=h_3\circ p(\gamma).
\end{gathered}
\end{equation}
This finishes the proof of the lemma.
\end{proof}

\section{Finishing the proof of Theorem \ref{t:main}}\label{proof}

Now we will construct a map of degree one from $M\rtimes_f S^1$ to $M_1\rtimes_h S^1$, for some homeomorphism $h\colon M_1\longrightarrow M_1$.

\medskip

As above, we can assume by Lemmas \ref{l:noaspherical} and \ref{l:asphericalsum} that $M=M_1\#\cdots\# M_n\#(\#_m S^2\times S^1)$, where $M_i$ are aspherical and $n\geq 1$. Moreover, we assume that $M_1$ is not interchanged under $f$ with another summand $M_i$ (by the proof of Lemma \ref{l:asphericalsum}).

Consider the classifying space $B\pi_1(M)=M_1\vee\cdots\vee M_n\vee(\vee_m S^1)$ and the homotopically unique map
\[
B(f_*)\colon M_1\vee\cdots\vee M_n\vee(\vee_m S^1)\longrightarrow M_1\vee\cdots\vee M_n\vee(\vee_m S^1),
\]
where $f_*\colon\pi_1(M)\longrightarrow\pi_1(M)$ is the isomorphism induced by $f$.
Let also the map
\[
B(p_*)\colon M_1\vee\cdots\vee M_n\vee(\vee_m S^1)\longrightarrow M_1,
\]
induced by the pinch map $p$. (Again, $p_*\colon \pi_1(M_1)\ast\cdots\ast\pi_1(M_n)\ast F_m\longrightarrow \pi_1(M_1)$ denotes the induced homomorphism.)

By Theorem \ref{Mc1} (and the comments after that), we know that $f=g_3g_2g_1$, where $g_3$ is a finite composition of homeomorphisms of type 3 and isotopies of $M$, and 
each of $g_1$, $g_2$ is a composition of finitely many homeomorphisms of type 1 and 2 respectively, as given in Section \ref{description}. 

By Lemma \ref{l:commutativity}, there is a homeomorphism $h\colon M_1\longrightarrow M_1$ such that 
\begin{equation}\label{commutativitygroup}
\begin{gathered}
p_*\circ f_*=h_*\circ p_*.
\end{gathered}
\end{equation}
For set $h=h_3h_2h_1$, where each (component of) $h_i$ is given by Lemma \ref{l:commutativity}. Then applying successively Lemma \ref{l:commutativity} on each $h_i$ we deduce that (\ref{commutativitygroup}) indeed holds.
Therefore, there is a well-defined surjective homomorphism
\[
\overline{p}_*\colon\pi_1(M\rtimes_f S^1)\longrightarrow\pi_1(M_1\rtimes_h S^1)
\]
which maps each element $x$ of $\pi(M)$ to $p_*(x)\in\pi_1(M_1)$ and the generator of the infinite cyclic group acting (through $f_*$) on $\pi_1(M)$ to the generator of the infinite cyclic group acting (through $h_*$) on $\pi_1(M_1)$. 

\begin{remark}\label{Nicommute}
If we replace $M_1$ by the connected sum $\#_m S^2\times S^1$ and adapt accordingly Lemma \ref{l:commutativity} to the situation of~\cite[Lemmas 3.6 and 3.7]{Ni}, then we will obtain a surjection 
\begin{equation}\label{Nisurjection}
\pi_1(M\rtimes_f S^1)\longrightarrow F_m\rtimes_{h_*}\Z.
\end{equation}
This surjection does not yield a map of non-zero degree, because the classifying space of $F_m$ is one dimensional.
\end{remark}

The homomorphism $\overline{p}_*$ gives rise to a well-defined map
\[
B(\overline{p}_*)\colon B\pi_1(M\rtimes_f S^1)\longrightarrow B\pi_1(M_1\rtimes_h S^1).
\]
Since $M_1$ is aspherical, the homotopy long exact sequence for $\pi_1(M_1\rtimes_h S^1)$ implies that $M_1\rtimes_h S^1$ is aspherical. Furthermore, 
\[
\pi_1(M\rtimes_f S^1)=\langle \pi_1(M), t \ | \ txt^{-1}=f_*(x), \ x\in\pi_1(M) \rangle=\pi_1(B\pi_1(M)\rtimes_{B(f_*)} S^1).
\]
Again, by the asphericity of $B\pi_1(M)$, we deduce that $B\pi_1(M)\rtimes_{B(f_*)} S^1$ is aspherical. Thus $B(\overline{p}_*)$ is (homotopic to) a map 
\begin{align*}
\begin{split}
B\pi_1(M)\rtimes_{B(f_*)} S^1 & \longrightarrow  M_1\rtimes_hS^1 \\
[(a,t)] & \ \mapsto \ [(B(p_*)(a),t)],
\end{split}
\end{align*}
which we still denote by $B(\overline{p}_*)$.
  
Define now a map
 \[
 F\colon M\rtimes_f S^1\longrightarrow M_1\rtimes_h S^1
 \]
by 
\[
F:=B(\overline{p}_*)\circ\psi_{M\rtimes_f S^1},
\]
where $\psi_{M\rtimes_f S^1}\colon M\rtimes_f S^1\longrightarrow B\pi_1(M)\rtimes_{B(f_*)} S^1$ is the classifying map for $M\rtimes_f S^1$ (recall that $\pi_1(M\rtimes_f S^1)=\pi_1(B\pi_1(M)\rtimes_{B(f_*)} S^1)$). Let also $\psi_M\colon M\longrightarrow B\pi_1(M)$ denote the classifying map for $M$. Then the following diagram
$$
\xymatrix{
0=H_4(M)\ar[d]^{0} \ar[r]^{0} &  H_4(M\rtimes_f S^1)\ar[d]^{H_4(\psi_{M\rtimes_f S^1})} \ar[r]^{\ \ \ \ \alpha_1} &  H_3(M) \ar[d]^{H_3(\psi_M)} \ar[r]^{0} &  H_3(M)\ar[d]^{H_3(\psi_M)}\ar[r]^{} & \cdots \\
0=H_4(B\pi_1(M))\ar[d]^{0} \ar[r]^{0 \ \ \ \ } &  H_4(B\pi_1(M)\rtimes_{B(f_*)} S^1)\ar[d]^{H_4(B(\overline{p}_*))} \ar[r]^{\ \ \ \ \ \ \alpha_2} &  H_3(B\pi_1(M)) \ar[d]^{H_3(B(p_*))} \ar[r]^{} &  H_3(B\pi_1(M))\ar[d]^{H_3(B(p_*))}\ar[r]^{} & \cdots  \\
0=H_4(M_1) \ar[r]^{0}& H_4(M_1\rtimes_h S^1)\ar[r]^{\ \ \ \ \alpha_3} &  H_3(M_1) \ar[r]^{} \ar[r]^{0} &  H_3(M_1)\ar[r]^{} & \cdots \\
}
$$
implies
\begin{equation*}\label{eq}
\begin{aligned}
\alpha_3\circ H_4(B(\overline{p}_*))\circ H_4(\psi_{M\rtimes_f S^1})([M\rtimes_f S^1])& = H_3(B(p_*))\circ\alpha_2\circ H_4(\psi_{M\rtimes_f S^1})([M\rtimes_f S^1])\\
                                         & =  H_3(B(p_*))\circ H_3(\psi_M)\circ\alpha_1([M\rtimes_f S^1])\\
                                         & =  H_3(B(p_*))\circ H_3(\psi_M)([M])\\
                                         & =  H_3(B(p_*))([M_1],...,[M_n])=[M_1].\\
\end{aligned}
 \end{equation*}
 This means that
 \[
 H_4(F)([M\rtimes_f S^1])=[M_1\rtimes_h S^1],
 \]
 completing the proof of Theorem \ref{t:main}.
 
\section{Virtual first Betti numbers}\label{virtual}

In this section we discuss Conjecture \ref{conjBetti}. 

\medskip

Recall that the finiteness of virtual first Betti numbers of mapping tori of prime $3$-manifolds follows that of their fiber $M$, namely $vb_1(M)$ (and $vb_1(M\rtimes_f S^1)$) is finite if and only if $M$ is virtually $S^3$, $S^2\times S^1$ or a $T^2$-bundle. More precisely, if a $3$-manifold $M$ is finitely covered by $S^3$, $S^2\times S^1$ or a $T^2$-bundle, then $vb_1(M)\leq 3$, and the corresponding mapping tori of $M$ satisfy $vb_1(M\rtimes_f S^1)\leq vb_1(M)+1\leq 4$ for any homeomorphism $f\colon M\longrightarrow M$. It is therefore natural to examine Conjecture \ref{conjBetti} according to whether a reducible $3$-manifold contains a prime summand with virtually infinite first Betti number or not.

\subsection{At least one prime summand with virtually infinite first Betti number}
Suppose first that a reducible $3$-manifold $M$ contains a summand in its prime decomposition with $vb_1=\infty$. This summand is necessarily aspherical. If $f\colon M\longrightarrow M$ is an orientation preserving homeomorphism, then Theorem \ref{Mc1} tells us that $f$ is isotopic to a composition $g_3g_2g_1$ where each $g_i$ is a finite composition of homeomorphisms of type $i=1,2,3$. By Lemmas \ref{l:noaspherical}, \ref{l:asphericalsum} and \ref{l:cover} (and their proofs), there is a finite a cover $\overline{M}$ of $M$ containing an aspherical summand $M_1$ in its prime decomposition which is not virtually a mapping torus of $T^2$ and is not interchanged by some $f^k$ (under a component of type 2) with any other summand of $\overline{M}$. Then Theorem \ref{t:main} implies that there is a self-homeomorphism $h$ of $M_1$ and a degree one map
\[
\overline{M}\rtimes_{f^k} S^1 \longrightarrow M_1\rtimes_h S^1.
\]
In particular, $ b_1(\overline{M}\rtimes_{f^k} S^1)\geq b_1(M_1\rtimes_h S^1)$, and so Theorem \ref{LiNi} implies that 
$$vb_1(M\rtimes_f S^1)=\infty.$$

This proves Conjecture \ref{conjBetti} in all cases, except when $M$ is virtually $(\#_n T^2\rtimes S^1)\#(\#_m S^2\times S^1)$ (containing at least two summands).

\subsection{Only summands with virtually finite first Betti numbers}

Suppose, finally, that $M$ is (virtually) of the form $(\#_n T^2\rtimes S^1)\#(\#_m S^2\times S^1)$. In that case, Theorem \ref{t:main} is not anymore applicable to deduce that $vb_1(M\rtimes_f S^1)=\infty$. On the one hand, if $n=0$, then $\#_m S^2\times S^1$ does not contain any aspherical summands. On the other hand, if $n\neq0$, then Theorem \ref{t:main} implies that there is a degree one map $\overline{M}\rtimes_{f^k} S^1\longrightarrow (T^2\rtimes S^1)\rtimes_h S^1$, which, however, does not suffice to conclude that $vb_1(M\rtimes_f S^1)=\infty$ because $vb_1((T^2\rtimes S^1)\rtimes_h S^1)\leq 4$.
It would be interesting to find topological arguments that cover those two cases as well. 

Nevertheless, we can appeal to group theoretic results to deduce that $vb_1=\infty$ in the remaining two cases. First, one can deduce from known results that $\pi_1((\#_m S^2\times S^1)\rtimes_f S^1)$ is large for $m\geq 2$: By~\cite{BF,Brink}, the free-by-cyclic group 
\[
F_m\rtimes_{f_*}\Z=\pi_1((\#_m S^1\times S^1)\rtimes_{f} S^1), \ m\geq 2,
\]
is word hyperbolic if and only if it does not contain as subgroup an isomorphic copy of $\Z^2$. In the case where $F_m\rtimes_{f_*}\Z$ is hyperbolic, then it is large by~\cite{Ag,HW,Wi}. If now $\Z^2\subset F_m\rtimes_{f_*}\Z$, then $F_m\rtimes_{f_*}\Z$ is large by~\cite{But}. Thus in all cases we deduce that $vb_1(F_m\rtimes_{f_*}\Z)=\infty$ as required.

Therefore, using the largeness of $F_m\rtimes_{f_*}\Z$, we conclude that $$vb_1((\#_m S^2\times S^1)\rtimes_f S^1)=\infty \text{ for } m\geq2.$$ 

\begin{remark}
By~\cite{BF,Brink}, $F_m\rtimes_{f_*}\Z$ being word hyperbolic is equivalent to the automorphism $f_*$ being atoroidal, i.e. having no non-trivial periodic conjugacy classes. When $f_*$ is toroidal (i.e. it has some non-trivial periodic conjugacy class), Ni also showed that $vb_1(F_m\rtimes_{f_*}\Z)=\infty$; see~\cite[Lemma 2.4]{Ni}.
\end{remark}

Finally, as Ni shows in all cases where aspherical summands exist, our remaining case of mapping tori of $(\#_n T^2\rtimes S^1)\#(\#_m S^2\times S^1)$, where $n\geq 1$, can be treated as follows: Recall that we can always assume that $m\geq 2$ and using the $\pi_1$-surjection induced by the pinch map
\[
\pi_1(((\#_n T^2\rtimes S^1)\#(\#_m S^2\times S^1))\rtimes_f S^1)\longrightarrow \pi_1((\#_m S^2\times S^1)\rtimes_h S^1),
\]
we deduce that $$vb_1(((\#_n T^2\rtimes S^1)\#(\#_m S^2\times S^1))\rtimes_f S^1)=\infty$$ as required; see also Remark \ref{Nicommute}.


\begin{thebibliography}{123}

\bibitem{Ag}
I. Agol, {\em The virtual Haken conjecture}, with an appendix by I. Agol, D. Groves, and J. Manning, Doc. Math., {\bf 18} (2013), 1045--1087.

\bibitem{BF}
M. Bestvina and M. Feighn, {\em A combination theorem for negatively curved groups}, J. Differential Geom., {\bf 35} (1992), 85--101.

\bibitem{Brink}
P. Brinkmann, {\em Hyperbolic automorphisms of free groups}, Geom. Funct. Anal., {\bf 10} (2000), 1071--1089.

\bibitem{Bon}
F. Bonahon, {\em Cobordism of automorphisms of surfaces}, Ann. Sci. \'Ecole Norm. Sup. {\bf 16} (1983), 237--270.

\bibitem{But}
J. O. Button, {\em Large groups of deficiency 1}, Isr. J. Math. {\bf 167} (2008), 111--140.

\bibitem{CR}
E. C\'esar de S\'a and C. Rourke, {\em The homotopy type of homeomorphisms of 3-manifolds}, Bull. Amer. Math. Soc. (N.S.) {\bf 1}, no. 1 (1979), 251--254. 

\bibitem{HW}
M. Hagen and D. Wise, {\em Cubulating hyperbolic free-by-cyclic groups: The general case}, Geom. Funct. Anal., {\bf 25} (2015), 134--179.

\bibitem{He}
J. Hempel, {\em 3-manifolds}, Princeton University Press And University of Tokyo Press, 1976.

\bibitem{KN}
D. Kotschick and C. Neofytidis, {\em On three-manifolds dominated by circle bundles}, Math. Z. {\bf 274} (2013), 21--32.

\bibitem{LiNi}
T.-J. Li and Y. Ni, {\em Virtual Betti numbers and virtual symplecticity of 4-dimensional mapping tori} Math. Z., {\bf 277} (2014), 195--208.

\bibitem{Mc} 
D. McCullough, {\em Mappings of reducible 3-manifolds}, Proceedings of the Semester in Geometric and Algebraic Topology, 
Warsaw, Banach Center (1986), 61--76.

\bibitem{NW}
C. Neofytidis and S. Wang, {\em Invariant incompressible surfaces in reducible 3-manifolds}, Ergodic Theory Dynam. Systems {\bf 39} (2019), 3136--3143.

\bibitem{Ni}
Y. Ni, {\em Virtual Betti numbers and virtual symplecticity of 4-dimensional mapping tori, II}, Sci. China Math. {\bf 60} no. 9 (2017), 1591--1598.

\bibitem{Z} X. Zhao, {\it On the Nielsen numbers of slide homeomorphisms on 3-manifolds.} Topology Appl. {\bf 136}, no. 1-3 (2004), 169--188. 

\bibitem{Wi}
D. Wise, {\sl From Riches to Raags: 3-Manifolds, Right-Angled Artin Groups, and Cubical Geometry}, CBMS Regional Conference Series in Mathematics {\bf 117}, American Mathematical Society, Providence, RI, 2012.

\end{thebibliography}
\end{document}